\documentclass[12pt,twoside]{article}

\usepackage{jltmac2e}
\usepackage{latexsym,amsmath,amssymb,graphics,graphicx,amscd,amsfonts}

\renewcommand{\L}{\mathop{\bf L{}}\nolimits}
\newcommand{\g}{{\mathfrak g}}
\newcommand{\h}{{\mathfrak h}}
\renewcommand{\:}{\colon}

\newcommand{\lie}[1]{\mathfrak{#1}}
\newcommand{\R}{\mathbb{R}}
\newcommand{\N}{\mathbb{N}}
\newcommand{\B}{\mathcal{B}}

\DeclareMathOperator{\im}{Im} 

\renewcommand{\version}{} 
\renewcommand{\vol}{18}     
\renewcommand{\annum}{2008} 
\renewcommand{\title}
{\vbox{\centerline{Closedness of the tangent spaces}
{\centerline{to the orbits of proper actions}}}}
\renewcommand{\author}{Madeleine Jotz, Karl-Hermann Neeb} 
\renewcommand{\lastname}{Jotz, Neeb}

\def\rec{April 23, 2008}
\def\fin{April 30, 2008} 

\renewcommand{\address}
{Madeleine Jotz 

D\'epartement de Math\'ematiques

Ecole Polytechnique F\'ed\'erale de Lausanne

1015 Lausanne

Switzerland

madeleine.jotz@epfl.ch}

\renewcommand{\addresstwo}
{Karl-Hermann Neeb 

Fachbereich Mathematik 

TU Darmstadt 

Schlossgartenstrasse 7 

64289 Darmstadt 

Germany 

neeb@mathematik.tu-darmstadt.de}

\begin{document}
\firstpage

\begin{abstract} 
In this note we show that for any proper action of a Banach--Lie group 
$G$ on a Banach manifold $M$, the corresponding tangent maps 
$\g \to T_x(M)$ have closed range for each $x \in M$, i.e., the 
tangent spaces of the orbits are closed. 
As a consequence, for each free proper action on a Hilbert manifold, 
the quotient $M/G$ carries a natural manifold structure. \\
\nin {\em Keywords and phrases:} Banach Lie group, Banach manifold, proper 
action 
 \hfill\break 
{\em Mathematics Subject Index 2000:} 22E65,  58B25,  57E20
\end{abstract}

%%%%%%%%%%
%%%%%%%%%%
\section{Introduction} 

It is a classical result of Palais that a proper smooth 
action of a non-compact finite-dimensional Lie group $G$ on a smooth 
manifold $M$ behaves in many  respects like the action of a compact 
Lie group. In particular, all orbits are closed submanifolds and the 
action of $G$ on a suitable neighborhood of an orbit can be described 
nicely in terms of the action of the compact stabilizer group on a 
slice (\cite{Pa61}). Since properness of a group action is a purely topological
concept, it also applies to smooth actions of Banach--Lie groups on 
Banach manifolds. However, one cannot expect results as strong 
as Palais' in this context. One reason is that a closed subspace 
of a Banach space need not have a closed complement, 
$c_0(\N) \subseteq \ell^\infty(\N)$ being a simple example. 
Since the translation action of the closed subspace 
$c_0(\N)$ on $\ell^\infty(\N)$ is proper, properness does not imply 
anything on the existence of closed complements. In particular, 
slices need not exist for proper smooth actions. 

Since the stabilizer subgroups $G_x$ are compact for a proper 
action, they are in particular finite-dimensional Lie subgroups, so that 
the quotient $G/G_x$ carries a natural manifold structure and the 
smooth injection $G/G_x \to Gx$
is a topological embedding. In this 
sense all orbits of a proper action carry a natural manifold structure. 
The main concern of this note is to show that this manifold structure 
is, at least in a very weak sense, compatible with that of the surrounding 
manifold $M$, namely the tangent space of $Gx$, i.e., the image 
of the tangent map $T_e(G) \to T_x(M)$ of the orbit map, 
is a closed subspace of $T_x(M)$. As a consequence, 
Palais' theory carries immediately over to proper actions on Hilbert 
manifolds. In the general case our result eliminates at least one 
of the assumptions, usually stated in addition to properness, for 
a slice theorem or a quotient theorem to hold 
(cf.\ \cite[Ch.~3, \S 1, no.~5, Prop.~10]{Bo89a}).

\section{Tangent maps of proper actions} 

Let $G$ be a Banach--Lie group with Lie algebra $\g = \L(G)$ 
and neutral element~$e$. 
In the following $M$ will be a smooth Banach manifold with a smooth
action $\Phi:G\times M\to M$ of $G$. 
The diffeomorphism of $M$ defined by $g\in G$ 
will be denoted $\Phi_g:M\to M$ and 
$\Phi(g,x)$ will simply be written $g\cdot x$ or $gx$. 
For each $x\in M$ we denote the \emph{orbit map} by $\Phi^x:G\to
M$, $g\mapsto gx$. 

To each $\xi\in\lie g$ we associate the corresponding 
vector field $\xi_M$ on $M$, defined by 
$\xi_M(x)=\left.\frac{d}{dt}\right\arrowvert_{t=0}\Phi_{\exp(t\xi)}(x)$ 
and recall that the linear map 
$\g \to {\cal V}(M), \xi \mapsto - \xi_M$ is a morphism of Lie algebras. 

The action is said to be \emph{proper} if the smooth map $\Theta:G\times
M\to M\times M$, $(g,x)\mapsto (gx,x)$, is proper, i.e., 
a closed map with compact fibers 
$$ \Theta^{-1}(x,y) = \{ (g,y) \in G \times M \: gy = x\}. $$
If $G$ and $M$ are finite-dimensional, i.e., locally compact, 
properness of the action is equivalent to inverse images of compact subsets 
$K \subseteq G \times M$ under $\Theta$ being compact. 

In this note we show the following theorem:
\begin{theorem}\label{thm}
If the action of $G$ on $M$ is  proper, then 
the image of the bounded linear map
$$ \Psi_x=T_e\Phi^x: \lie g\to T_xM, \quad \xi\mapsto \xi_M(x) $$
is closed for all $x\in M$.
\end{theorem}

For the proof we need the following three lemmas. With the first lemma
we are able to describe the kernel of the map
$T_e\Phi^x$ as the Lie algebra of the isotropy group of $x$, for an arbitrary
$x\in M$. 

\begin{lemma}\label{lem:1}
For each $x\in M$ the following assertions hold: 
\begin{description}
\item[\rm(a)] The isotropy subgroup $G_x:=\{g\in G\mid gx=x\}$ 
is a compact Lie subgroup of $G$. In particular, 
its Lie algebra 
$\g_x=\{\xi\in\lie g \mid \xi_M(x)=0\}$ is finite-dimensional, 
hence complemented. The quotient space $G/G_x$ carries a natural  
manifold structure such that for each closed complement 
$\h_x$ of $\g_x$ in $\g$ the map 
$\h_x \to G/G_x, \xi \mapsto \exp(\xi)G_x$ 
is a local diffeomorphism in $0$. 
\item[\rm(b)] The orbit map $\Phi^x \: G \to M$ induces a topological 
embedding $G/G_x \to M$ onto a closed subset. 
\end{description}
\end{lemma}

\begin{proof} (a) It follows from the properness of the action that the set
$\Theta^{-1}(x,x)=G_x\times \{x\}$ is compact. Hence $G_x$ is a compact 
subgroup of $G$ and therefore a finite-dimensional Lie subgroup 
(cf.\ \cite[Thm.~IV.3.16]{Ne06}; \cite{GN08} and \cite[Thm.~5.28]{HM98}) 
whose Lie algebra is 
$$ \g_x = \{ \xi \in \g \: \exp(\R \xi) \subseteq G_x \} 
= \{ \xi \in \g \: \xi_M(x)=0\} = \ker(\Psi_x). $$
Since the finite-dimensional subspace $\g_x$ of $\g$ has a closed 
complement, $G_x$ is a split Lie subgroup 
(called a Lie subgroup in \cite{Bo89a}), so that the remainder 
of (a) follows from \cite[Ch.~3, \S 1.6, Prop.~11]{Bo89a}. 

\nin (b) This is a general property of proper group actions 
(\cite[Ch.~3, \S 4.2, Prop.~4]{Bo89b}). 
\end{proof}

The second lemma is a well-known  fact in functional analysis and can be found for example in \cite[Lemma 4.47]{RY08}. 

\begin{lemma}\label{lem:2}
Let $X$ and $Y$ be Banach spaces and $T\: X \to Y$ be continuous linear. 
If there exists an $\alpha>0$ such that
$\|Tx\|\geq \alpha\|x\|$ for all $x\in X$, then $\im(T)$ is closed.
\end{lemma}

The third lemma we will need can be found in \cite[IV,\S 1, no.~4]{Bo51}:

\begin{lemma}\label{lem:3} Let $I$ be an interval and $W$ an open subset of a
  normed space~$E$.
Let $k:I\to\mathbb{R}_{>0}$ be a real regulated function and 
$f:I\times W\to E$ a function satisfying $\|f(t,x_1)-f(t,x_2)\|\leq
k(t)\|x_1-x_2\|$ for all $t\in I$ and
$x_1,x_2\in W$.
Let $u_i:I\to W$, $i =1,2$, be two differentiable function 
which are approximate solutions of $\dot{c}(t)=f(t,c(t))$ in the sense 
that $\|\dot u_i(t)-f(t,u_i(t))\|\leq \varepsilon_i$ 
for all $t\in I$ and $i =1,2$. Then we have 
for  $t\in I$ with $t\geq t_0$:
\begin{equation*}
\|u_1(t)-u_2(t)\|\leq \|u_1(t_0)-u_2(t_0)\|\Lambda(t)
+(\varepsilon_1+\varepsilon_2)\Pi(t)
\end{equation*}
where 
\begin{equation*}
\Lambda(t)=1+\int_{t_0}^tk(s)\exp\left(\int_s^tk(\tau)d\tau\right)ds
\end{equation*}
and 
\begin{equation*}
\Pi(t)=t-t_0+\int_{t_0}^t(s-t_0)k(s)\exp\left(\int_s^tk(\tau)d\tau\right)ds.
\end{equation*}
\end{lemma}

\begin{proof}[of Theorem~\ref{thm}]
Let $\h_x \subseteq \g$ be a closed complement of $\g_x$ 
(see Lemma~\ref{lem:1}(a)). 
There exists an open neighborhood $V$ of $0\in\h_x$ such that the map \break 
$\gamma \: \h_x \to G/G_x, \xi \mapsto \exp(\xi)G_x$, 
is a diffeomorphism onto an open subset of $G/G_x$. 

Choose $r>0$ such that the closed ball $\overline{B_r(0)}$ is
contained in $V$. Since the orbit map $\Phi^x$ defines an embedding 
of $G/G_x$ into $M$, for each sequence $(\xi_n)$ in 
 $\lie h_x$ with $\|\xi_n\|=r$ for all $n\in \N$, the sequence
$(\exp(\xi_n)\cdot x)$ cannot converge to $x\in M$ because 
$\exp(\xi_n)G_x$ does not converge to $G_x$ in $G/G_x$ 
(Lemma~\ref{lem:1}(b)). 

For any  $x\in M$, the image of $\Psi_x$ 
is $\Psi_x(\lie h_x)$ since $\lie g_x = \ker \Psi_x$ is its kernel 
(Lemma~\ref{lem:1}(a)). 
Since $\lie h_x$ is a closed subspace of $\lie g$, it is a Banach space 
and we consider the injective map 
\begin{align*}
A_x:=\left.\Psi_x\right|_{\lie h_x}: \lie h_x&\to T_xM, \quad 
\xi\mapsto \xi_M(x). 
\end{align*}
If $\Psi_x(\g) = A_x(\h_x)$ is not closed in 
$T_xM$, then with the aid of Lemma \ref{lem:2} we find for each $n \in \N$ 
an element $\xi_n\in\lie h_x$ such that $\|A_x(\xi_n)\|<\frac{1}{n}\|\xi_n\|$, and hence 
\[\left\|A_x\left(\frac{\xi_n}{\|\xi_n\|}\right)\right\|<\frac{1}{n}.\]
Thus, setting $\eta_n:=\frac{r\xi_n}{\|\xi_n\|}$ we get a sequence $(\eta_n)$
in $\lie g$ such that $\|\eta_n\|=r$ for all $n$ and
\[\lim_{n\to\infty}A_x(\eta_n)=0.\]

Let $\phi:U\to W\subseteq E$ be a local chart for $M$ centered on the point
$x$  (that is, $\phi(x)=0$) and 
$\xi_W := \phi_*(\xi_M\vert_U)$ denote the corresponding smooth vector 
field on the open subset $W$ of $E$. 
For two Banach spaces $X$ and $Y$ we write $\B(X,Y)$ for 
the space of continuous linear operators $X \to Y$. 
Then we obtain a smooth map 
$$\Gamma: W \to \B(\lie g,\B(E)),\quad \Gamma_z(\xi) := T_z(\xi_W) $$
(that can be understood as 
the smooth map $z\mapsto (\xi \mapsto D_1 F(z,\xi))$, 
where $F(z,\xi)=D_1\Phi(e,z)\xi$). 
Since $\Gamma$ is smooth, it is in particular continuous in $0$, 
so that we find for each $\varepsilon>0$ a $\delta>0$ with 
$\Gamma(B_\delta(0))\subseteq B_\varepsilon(\Gamma_0)$. 
Replacing $W$ by $B_\delta(0)$, we get for each $z\in W$: 
\[\|\Gamma_z\|\leq\|\Gamma_0\|+\varepsilon=:L.\]
Setting $D := L \cdot r$, this yields for all $x_1$ and $x_2\in W$ and all $\xi\in\lie g$ with $\|\xi\|\leq r$:
\begin{align*}
\|\xi_W(x_1)-\xi_W(x_2)\|&\leq \underset{z\in W}\sup\|T_z\xi_W\|\cdot\|x_1-x_2\|
\leq\underset{z\in W}\sup(\|\Gamma_z\|\cdot\|\xi\|)\cdot\|x_1-x_2\|\\
&\leq L \cdot\|\xi\|\cdot\|x_1-x_2\|
\leq L \cdot r\cdot\|x_1-x_2\|
= D\cdot\|x_1-x_2\|. 
\end{align*}
We now shrink $r$ such that the map  
$$ \xi \mapsto \phi(\exp(\xi)\cdot x) $$ 
maps the closure of the ball $B_r(0)\subseteq \lie g$ 
into the ball $W = B_\delta(0)$. 
Then for each $\xi \in \g$ with $\|\xi\| \leq r$ 
the curve $\gamma^\xi:[0,1]\to W, t \mapsto \phi(\exp t \xi \cdot x)$ 
is part of the integral curve of the vector field $\xi_W$ through $0 = \phi(x)$
(thus, $\gamma^\xi$ can be considered as an approximate 
solution of $\dot c(t)=\xi_W(c(t))$ for $\varepsilon_1 = 0$).
Define $\gamma(t)=0$ for all $t\in [0,1]$ and note that $\gamma$ is a 
approximate 
solution of $\dot c(t)=\xi_W(c(t))$ for $\varepsilon_2 := \|\xi_W(0)\|$.
Applying Lemma \ref{lem:3} to $I=[0,1]$, $t_0=0$,  $f:I\times W\to E$,
$f(z,t)=\xi_W(z)$, the constant function $k=D$ and the curves $\gamma$  and
$\gamma^\xi$,  we get
\[\|\gamma^\xi(t)-0\|\leq
\|0-0\|\cdot\Lambda(t)+\|\xi_W(0)\|\cdot\Pi(t)=\|\xi_W(0)\|\cdot\Pi(t)\]
with $\Pi(t)=t+\left[-\frac{Ds+1}{D}\exp((t-s)D)\right]_0^t=\frac{1}{D}(\exp(Dt)-1)$.
Hence, we have 
\begin{align*}
\|\phi(\exp(\xi)\cdot x)\|=\|\gamma^\xi(1)\|\leq \Pi(1)\cdot\|\xi_W(0)\|.
\end{align*} 
Thus, since $\Pi$ doesn't depend on $\xi$, we get for all $n\in \mathbb{N}$:
\begin{align*}
\|\phi\big(\exp (\eta_n)\cdot x\big)\|\leq \Pi(1)\cdot\|(\eta_n)_W(0)\|
\end{align*}
and because we have $\lim_{n\to\infty}(\eta_n)_W(0)=0$, we obtain 
$\lim_{n\to\infty}\exp (\eta_n)\cdot x=x$. We have seen above that this is a
contradiction to the properness of the action.
\end{proof}

Combining \cite[Ch.~3, \S 1, no.~5, Prop.~10]{Bo89a} with Theorem~\ref{thm}, 
we immediately obtain: 

\begin{corollary} If $M$ is a Hilbert manifold and 
$\Phi \: G \times M \to M$ a smooth action of a Banach--Lie 
group which is free and proper, then the set $M/G$ of \break $G$-orbits 
in $M$ carries a unique Hilbert manifold structure for which 
the quotient map $M \to M/G$ is a submersion. 
\end{corollary}

\begin{remark} In \cite{AN07} it is shown that for any smooth 
action $G \times M \to M$ of a Banach--Lie group and $x \in M$, the stabilizer 
subgroup $G_x$ is a (not necessarily split) Lie subgroup of $G$ 
if the map $\Psi_x \: \g \to T_x(M)$ has closed 
range. This result demonstrates the interesting interplay between 
properties of smooth actions and closedness conditions for the 
maps $\Psi_x$. For proper actions, stabilizers are compact, 
hence even split Lie subgroups (Lemma~\ref{lem:1}).  

It is an open problem whether stabilizer subgroups $G_x$ of smooth 
actions are always Lie subgroups. That they need not be split 
is an immediate consequence of the fact that for each 
non-split closed subspace $F$ of a Banach space $G = X$, the 
subspace $F$ is the stabilizer group $G_x$ of any point $x$ 
in the quotient space $M = X/F$ with respect to the induced action.  
\end{remark}

\nin{\bf Acknowledgement:} We thank Tudor Ratiu for 
drawing our attention to the problem and 
for interesting discussions on the subject. 
We also thank H.~Gl\"ockner for some comments on earlier version 
of the manuscript and the referee for several useful remarks and 
pointing out a reference to Lemma~\ref{lem:3}.

\lastpage

\end{document}